\documentclass[11pt,a4paper,sort&compress]{article}
\usepackage{amssymb,amsmath,amsthm}
\newtheorem{theorem}{Theorem}[section]
\newtheorem{corollary}[theorem]{Corollary}
\newtheorem{definition}[theorem]{Definition}
\newtheorem{example}[theorem]{Example}
\newtheorem{lemma}[theorem]{Lemma}

\newtheorem{remark}[theorem]{Remark}

\def\B{{\mathcal{B}}}
\def\R{{\mathcal{R}}}
\def\M{{\mathcal{M}}}
\def\N{{\mathcal{N}}}

\def\H{{\mathcal{H}}}

\def\mf{\mathfrak}
\newcommand{\cl}[1]{\overline{#1}}

\begin{document}

%\hskip5cm\includegraphics{memo.eps}
%
%\vskip-2.5cm

%%%%%% TO BE ENTERED BY THE AUTHOR(S)
%%%
%%% ENTER TITLE
\title{Operators with compatible ranges\footnote{The manuscript is accepted for publication in FILOMAT}}

%%% AUTHOR(S) FULL NAMES, AND EMAIL ADDRESSES
\author{Marko S. Djiki\'c\footnote{The author is supported by Grant No. 174007 of the  Ministry of Education, Science and Technological Development, Republic of Serbia.}\\
{\small Faculty of Sciences and Mathematics, University of Ni\v s, Serbia}
\\
{\small E-mail address: \textsf{marko.djikic@gmail.com}}}

%%%

%%% ENTER AUTHOR(S) AFFILIATION(S)
%\address[affil1]{Faculty of Sciences and Mathematics, University of Ni\v s, Vi\v segradska 33, 18000 Ni\v s, Serbia}

%%% AND CORRESPONDINGLY FOR OTHER AUTHORS, IF THERE ARE MORE AUTHORS
%%% ENTER ABBREVIATED AUTHOR(S) NAMES FOR PAGE HEADINGS
%\newcommand{\AuthorNames}{M. S. Djiki\'c}
%%% IF THERE ARE MORE THAN TWO AUTHORS WRITE
%%% \newcommand{\AuthorNames}{First Author et al.}
%%%

%%% ENTER MSC, KEYWORDS, RECEIVED, EDITOR, THANKS FOR FINANCIAL SUPPORT FOR RESEARCH
%\newcommand{\FilMSC}{Primary 47A05; Secondary 15A09}
%\newcommand{\FilKeywords}{CoR operators, Moore-Penrose inverse}
%\newcommand{\FilCommunicated}{(name of the Editor, mandatory)}
%\newcommand{\FilSupport}{The author is supported by Grant No. 174007 of the  Ministry of Education, Science and Technological Development, Republic of Serbia.}
%%% If you do not want to thank for the financial support of the research, remove
%%% the previous line (i.e., leave \FilSupport undefined)
%%%%%%%%%%%%%%%%%%%%%%%%%%%%%%%%%%%%%%%%%%%%%%%%%%

\maketitle

\begin{abstract}
A bounded operator $T$ on a finite or infinite--dimensional Hilbert space is called a disjoint range (DR) operator if $\R(T)\cap \R(T^*)=\{0\}$, where $T^*$ stands for the adjoint of $T$, while $\R(\cdot)$ denotes the range of an operator. Such operators (matrices) were introduced and systematically studied by Baksalary and Trenkler, and later by Deng et al. In this paper we introduce a wider class of operators: we say that $T$ is a compatible range (CoR) operator if $T$ and $T^*$ coincide on $\R(T)\cap \R(T^*)$. We extend and improve some results about DR operators and derive some new results regarding the CoR class.

\end{abstract}

%%%%%% THIS PART MUST BE PLACED IMMEDIATELY AFTER THE \maketitle COMMAND
%%%%%% BACK TO ORIGINAL FOOTNOTES
%%%%%%

\section{Motivation and preliminaries}

It is well-known that interesting properties of a real or complex square matrix $A$ can be described through certain geometric relations between its column space and the column space of its adjoint matrix $A^*$. For example, the column spaces $\R(A)$ and $\R(A^*)$ coincide if and only if the matrix $A$ commutes with its Moore-Penrose generalized inverse $A^\dag$. Such matrices are known as EP matrices, and they were the subject of many research papers (see also \cite[Chapter 4]{BenIsrael}). Quite opposite, if $\R(A)\oplus \R(A^*)$ is equal to whole space, then and only then $AA^\dag - A^\dag A$ is nonsingular (hereafter, $\oplus$ denotes the direct, not necessarily orthogonal, sum). Such matrices are called co-EP matrices, and they were introduced and studied by Ben\'itez and Rako\v cevi\'c \cite{Benitez1}. Werner \cite{Werner} studied the pairs of matrices $A$ and $B$ with conveniently positioned column spaces: $\R(A)\cap \R(B)=\{0\}$ and $\R(A^*)\cap \R(B^*)=\{0\}$. It turns out that such matrices are particulary useful with joint systems of equations $Ax=a$, $Bx=b$, etc.

As a generalization of a class of co-EP matrices, Baksalary and Trenkler \cite{Baks1} introduced a new class of matrices which merits its own name: \emph{disjoint range} matrices. A matrix $A\in \mathbb{C}^{n\times n}$ is said to be a disjoint range (or DR) matrix if $\R(A)\cap \R(A^*)=\{0\}$. They proved many properties of such matrices, of their functions and appropriate Moore-Penrose inverses. However, their study was based on linear algebra techniques, which are not appropriate for infinite-dimensional Hilbert spaces. The study of DR matrices, i.e. operators on arbitrary Hilbert spaces was conducted by Deng et al. \cite{Deng2}. Among others, the authors in \cite{Deng2} studied the classes of operators described in the following definition.
\begin{definition}\label{djd1}
Let $\H$ be a Hilbert space, and $T$ a bounded linear closed range operator on $\H$. Then $T$ is:
\begin{itemize}
\item[1)] DR if $\R(T)\cap \R(T^*)=\{0\}$;
\item[2)] EP if $\R(T)=\R(T^*)$;
\item[3)] SR if $\R(T)+\R(T^*)=\H$;
\item[4)] co-EP if $\R(T)\oplus \R(T^*)=\H$;
\item[5)] weak-EP if $P_{\cl{\R(T)}}P_{\cl{\R(T^*)}} = P_{\cl{\R(T^*)}}P_{\cl{\R(T)}}$,
\end{itemize}
where $P_{\M}$ denotes the orthogonal projection onto a closed subspace $\M$.
\end{definition}
%If we consider matrices as linear operators on appropriate finite-dimensional spaces, it is clear that the Definition \ref{djd1} extends notions of DR, EP, etc. from matrices to operators on arbitrary Hilbert spaces. We kept the condition of closed range of $T$ in Definition \ref{djd1} as the authors did in \cite{Deng2}.

However, one very important class of operators is not fully contained in the union of the classes from Definition \ref{djd1}. Namely, if $P$ and $Q$ are two orthogonal projections on a Hilbert space $\H$, the operator $PQ$ need not to belong to any of the mentioned classes, and not only because its range need not to be closed (see Example \ref{djex1}). This is our main motivation to extend the DR class in the following way. Note that we do not ask for $T$ to have a closed range, although most of the presented results will deal with closed range operators.

\begin{definition}\label{djd2} Let $\H$ be a Hilbert space, and $T$ a bounded linear operator on $\H$. We say that $T$ is a \emph{compatible range operator} (CoR)  if $T$ and $T^*$ coincide on the set $\cl{\R(T)}\cap \cl{\R(T^*)}$.
\end{definition}
The second reason for such generalization comes from the results of \cite{Djikic, Djikic2, Djikic3, Djikic4} which describe different properties of operators $A$ and $B$ which coincide on $\cl{\R(A^*)}\cap \cl{\R(B^*)}$ (a generalization of Werener's condition of weak complementarity, see \cite{Werner}). Accordingly, we will present different properties of CoR operators, regarding range additivity, some additive results for the Moore-Penrose inverse, etc. We also extend and improve some properties of DR operators, and in the end we give a discussion regarding operators that are products of orthogonal projections.

 In the rest of this section we introduce some notation which is not yet mentioned, and we recall some notions. Throughout the paper, $\H$ will stand for an arbitrary complex Hilbert space, that can also be infinite-dimensional. The algebra of bounded operators on $\H$ will be denoted by $\B(\H)$, and if $A\in \B(\H)$ then $\N(A)$ stands for the null-space of $A$. If $\H = \M \oplus \N$, where $\M$ and $\N$ are closed subspaces, we say that this decomposition completely reduces $A$ if $\M$ and $\N$ are invariant subspaces for $A$. In that case, $A$ is an isomorphism if and only if the reductions of $A$ on $\M$ and $\N$ are both isomorphisms. We write $\M\ominus \N$ to denote $\M\cap \N^\bot$. If $A\in \B(\H)$ is such that $\R(A)$ is closed, then the following system of equations:
$$AXA=A, \quad XAX=X,\quad (AX)^*=AX,\quad (XA)^*=XA,$$ has a unique solution $A^\dag$, which is called the Moore-Penrose generalized inverse of $A$. In that case $AA^\dag$ is the orthogonal projection onto $\R(A)$ and $A^\dag A$ is the orthogonal projection onto $\R(A^*)$.

If for $A\in \B(\H)$ stands $\R(A)\oplus \N(A)=\H$ then the following system of equations has a unique solution: $AXA=A,\quad XAX=X,\quad XA=AX.$ Such $A$ is called group-invertible, and the solution of this system is called the group inverse of $A$. In \cite{Deng2}, the list of classes of operators to be studied contains the group-invertible operators as well. We did not include it in Definition \ref{djd1} since this class is not defined through interrelation between the ranges of an operator and its adjoint. We should however mention that the product of two orthogonal projections, provided the range is closed, is indeed group-invertible (see \cite[Theorem 4.1]{Corach3}). We also wish to emphasize that we abbreviated \emph{compatible range} as CoR, and not as CR, since CR is commonly used for the class of closed range operators.

\section{CoR operators}

The main framework for studying DR matrices and DR operators was established through certain space and operator decompositions. In \cite{Baks1} the Hartwig-Spindelb\"ock decomposition of matrix is used (see \cite{Hartwig4}), and in case of operators on arbitrary Hilbert space, the appropriate operator decomposition is used: if $T\in \B(\H)$ then
\begin{equation}\label{djeq1}
T = \begin{bmatrix} A & B \\ 0 & 0\end{bmatrix}: \begin{bmatrix}\cl{\R(T)}  \\ \N(T^*)\end{bmatrix} \to \begin{bmatrix} \cl{\R(T)}  \\ \N(T^*)\end{bmatrix}.
\end{equation}
The reader is referred to \cite[Lemma 1.2]{Djordjevic1} and the discussion therein for further properties of such decompositions.

If $T$ is a closed range operator, \cite[Theorem 3.5]{Deng2} gives necessary and sufficient conditions for $T$ to be DR, SR and co-EP operator, under the additional assumption that $\R(TT^\dag - T^\dag T)$ is closed (which will be the subject of Lemma \ref{djl1}). The main tool in that proof is the famous Halmos' two projections theorem (see \cite{Bottcher1, Halmos1}). However, this assumption is dispensable if we apply a more direct approach.

\begin{theorem}\label{djt2} Let $T\in \B(\H)$ be a closed range operator, with operators $A$ and $B$ defined as in \eqref{djeq1}. Then:
\begin{itemize}
\item[(1)] $T$ is DR if and only if $\cl{\R(B)} = \R(T)$;
\item[(2)] $T$ is SR if and only if $\R(B^*)=\N(T^*)$;
\item[(3)] $T$ is co-EP if and only if $B$ is invertible.
\end{itemize}
\end{theorem}

\begin{proof}
(1) Since
$$T^* = \begin{bmatrix} A^* & 0 \\ B^* & 0\end{bmatrix}: \begin{bmatrix} \R(T) \\ \N(T^*)\end{bmatrix} \to \begin{bmatrix} \R(T) \\ \N(T^*)\end{bmatrix}$$
then $\R(T)\cap \R(T^*)=\{0\}$ if and only if for every $x\in \R(T)$ the implication $B^*x=0 \Longrightarrow A^* x=0$ holds. This is equivalent with $\cl{\R(A)}\subseteq \cl{\R(B)}$, which is equivalent with $\cl{\R(A)}+\cl{\R(B)} = \cl{\R(B)}$.

The subspace $\R(T)$ is closed and $\R(T)=\R(A)+\R(B)$, so we have $\cl{\R(A)} + \cl{\R(B)}\subseteq \cl{\R(A)+\R(B)} =\cl{\R(T)}=\R(T)=\R(A) + \R(B)$. Hence $\cl{\R(A)} + \cl{\R(B)} = \cl{\R(B)}$ if and only if $\cl{\R(B)}=\R(T)$ and the statement (1) is proved.
\\
\\
(2) First let us prove that $\R(T)+\R(T^*)=\R(T) \oplus \R(B^*)$. For every $x\in \R(T)$ we have that $B^* x = - A^*x + (A^*x + B^*x) $, where $A^*x\in \R(T)$ and $A^*x + B^*x\in \R(T^*)$. Thus $\R(B^*)\subseteq \R(T)+\R(T^*)$ and so $\R(T)\oplus \R(B^*)\subseteq \R(T)+\R(T^*)$. The other implication is clear, since $\R(T^*)\subseteq \R(T)\oplus \R(B^*)$. Thus $\H=\R(T) + \R(T^*)$ if and only if $\H = \R(T)\oplus \R(B^*)$, and $\R(B^*)\subseteq \N(T^*)$, so this is equivalent with $\R(B^*)=\N(T^*)$.
\\
\\
(3) If $T$ is co-EP then $T$ is DR and SR, so $\cl{\R(B)}=\R(T)$ and $\R(B^*)=\N(T^*)$. Thus $\R(B^*)$, i.e. $\R(B)$ is closed, $\R(B)=\R(T)$ and $\N(B)=\R(B^*)^\bot = \{0\}$, showing that $B$ is invertible.

If $B$ is invertible, from (1) and (2) we conclude that T is in the same time DR and SR, so it is co-EP.
\end{proof}

%Unlike in \cite[Theorem 3.5]{Deng2}, in statement (1) of Theorem \ref{djt2} we have to keep the closure of the range of $B$, since it does not have to be closed. We prove the following lemma in order to shed more light on the condition of closed range $\R(TT^\dag - T^\dag T)$. In Theorem \ref{djt4} we will describe this condition  in even  more detail.

\begin{lemma}\label{djl1} If $T\in \B(\H)$ is a closed range operator, then $\R(TT^\dag - T^\dag T)$ is closed if and only if $\R(T)+\R(T^*)$ is closed, if and only if $\R(B)$ is closed, where $B$ is as in \eqref{djeq1}.
\end{lemma}
\begin{proof}
Operators $TT^\dag$ and $T^\dag T$ are orthogonal projections, so from \cite[Lemma 2.4]{Koliha1} we have that $\R(TT^\dag - T^\dag T)$ is closed iff $\R(TT^\dag) + \R(T^\dag T)$ is closed, iff $\R(T)+\R(T^*)$ is closed.

As in the proof of statement (2) in Theorem \ref{djt2} we have that $\R(T)+\R(T^*)=\R(T)\oplus \R(B^*)$. Since $\R(T)$ is closed and $\R(B^*)\subseteq (\R(T))^\bot$ we have that $\R(T)\oplus \R(B^*)$ is closed iff $\R(B^*)$ is closed, i.e. iff $\R(B)$ is closed.
\end{proof}

It is clear from Lemma \ref{djl1} that \cite[Theorem 3.5, (i)]{Deng2} follows from Theorem \ref{djt2}, while the other statements of \cite[Theorem 3.5]{Deng2} hold verbatim without additional assumptions.

A natural connection between CoR and DR  operators is described by the following statements.

\begin{lemma}\label{djl2}
Let $T\in \B(\H)$ be a closed range CoR operator. Then $T(\R(T)\cap \R(T^*))=\R(T)\cap \R(T^*)$, $T(\R(T^*)\ominus (\R(T)\cap \R(T^*)))=\R(T)\ominus (\R(T)\cap \R(T^*))$, and consequently, $T((\R(T)\cap \R(T^*))^\bot)\subseteq (\R(T)\cap \R(T^*))^\bot$.
\end{lemma}
\begin{proof}
Follows directly from \cite[Lemma 2.1]{Djikic2}, applied to $T$ and $T^*$.
\end{proof}

\begin{theorem}\label{djt3}
Let $T\in \B(\H)$ be a closed range CoR operator. There exists a Hilbert space $\H_1$, a bounded linear surjection $\pi: \H \to \H_1$ and an operator $T_1\in \B(\H_1)$ such that:
 \begin{itemize}
 \item[(1)] $T_1$ has a closed range and it is DR;
 \item[(2)] For every $x\in \H$, $\pi(T x)=T_1 \pi(x)$, and  $\pi(T^* x) =T_1^* \pi(x)$;
 \item[(3)] $\N(\pi)=\R(T)\cap \R(T^*)$;
 \item[(4)] For every $x\in \H$, $||\pi(x)||=||(I-P_{\R(T)\cap \R(T^*)}) x||$.
 \end{itemize}
If $\pi$ satisfies these conditions, and $\M$ is a subspace of $\H$ such that $\R(T)\cap \R(T^*)\subseteq \M$ then $\M$ is closed in $\H$ if and only if $\pi(\M)$ is closed in $\H_1$.
\end{theorem}
\begin{proof}
Let $\H_1$ be the orthogonal complement of $\R(T)\cap \R(T^*)$ in $\H$ and $\pi: \H\to \H_1$ defined as $\pi(x)=(I - P_{\R(T)\cap \R(T^*)})x$. In that case $\pi$ is a bounded linear surjection which satisfies (3) and (4).

Usgin Lemma \ref{djl2} it is not difficult to see that the operator $T_1:\H_1 \to \H_1$ defined as $T_1 x = T x$, for every $x\in \H_1$, is a well--defined operator, with a closed range, satisfying all the given conditions. This is easily seen from $T((\R(T)\cap\R(T^*))^\bot)\subseteq (\R(T)\cap\R(T^*))^\bot$, $\R(T_1)=\pi(\R(T))=\R(T)\ominus (\R(T)\cap \R(T^*))$, etc.

To prove the last statement, note that if $\M$ is a closed subspace of $\H$ such that $\R(T)\cap \R(T^*)\subseteq \M$ and $\N=(\R(T)\cap \R(T^*))^\bot \cap \M$, then $\M = (\R(T)\cap \R(T^*))\oplus \N$, $\N\bot \R(T)\cap \R(T^*)$ and according to (4), $\pi$ is an isometry on $\N$. So $\M$ is closed iff $\N$ is closed, iff $\pi(\N)$ is closed, iff $\pi(\M)$ is closed, since according to (3) we have $\pi(\N)=\pi(\M)$.
\end{proof}

\begin{remark}\label{djr1} %The operator $T$ from the previous theorem need not to be a closed range operator. We should only change $\R(T)$ with $\cl{\R(T)}$ and similarly for $\R(T^*)$ in every appearance of these subspaces in the statement and in the proof. In that case, operator $T_1$ would not be a closed range operator, but it would hold $\cl{\R(T_1)}\cap \cl{\R(T_1^*)} = \{0\}$.
The converse of Theorem \ref{djt3} is not true: if there exist such $\H_1, \pi$ and $T_1$, the operator $T$ need not to be CoR. However, in that case we can conclude that $T(\R(T)\cap \R(T^*))\subseteq \R(T)\cap \R(T^*)$, and similarly for $T^*$, and so the decomposition $\H=(\R(T)\cap \R(T^*))\oplus (\R(T)\cap \R(T^*))^\bot$ completely reduces both $T$ and $T^*$. This further yields that $T$ is an isomorphism on $\R(T)\cap \R(T^*)$ although it is not necessarily self--adjoint.
\end{remark}

 In order to state the characterization of CoR operators similar to that in Theorem \ref{djt2}, let:
 \begin{equation}\label{djeq2}
 P_{\cl{\R(T)}\cap \cl{\R(T^*)}} = \begin{bmatrix} P & 0 \\ 0 & 0 \end{bmatrix}: \begin{bmatrix} \cl{\R(T)}  \\ \N(T^*)\end{bmatrix} \to \begin{bmatrix} \cl{\R(T)}  \\ \N(T^*)\end{bmatrix},
 \end{equation} where $P\in \B(\cl{\R(T)})$ is the orthogonal projection with the range $\cl{\R(T)}\cap \cl{\R(T^*)}$ and the null-space $\cl{\R(T)}\ominus \cl{\R(T)}\cap \cl{\R(T^*)}$. Also, if $\R(T)$ is closed and $A$ and $B$ defined as in \eqref{djeq1}, then $AA^* + BB^*\in \B(\R(T))$ is invertible, and as in \cite{Deng2} we denote $\Delta = (AA^* + BB^*)^{-1}$.

\begin{theorem}\label{djt5} Let $T\in \B(\H)$ be an operator with $A$ and $B$ defined as in \eqref{djeq1}. The operator $T$ is CoR if and only if $A$ and $A^*$ coincide on $\cl{\R(T)}\cap \cl{\R(T^*)}$ and $\cl{\R(T)}\cap \cl{\R(T^*)}\subseteq \N(B^*)$. In that case, we have:
\begin{itemize}
%\item[(1)] $T(\R(T)\cap \R(T^*))=\R(T)\cap \R(T^*)$ and $T(\R(T)\ominus (\R(T)\cap \R(T^*)))=\R(T)\ominus (\R(T)\cap \R(T^*))$;
\item[(1)] $PAP=AP=A^*P = PA^*P$. If $\R(T)$ is closed, then also $P\Delta P = \Delta P$;
%\item[(3)] $A$ is CoR and $\cl{\R(A)}\cap \cl{\R(A^*)} = \R(T)\cap \R(T^*)$;
\item[(2)] $\N(B^*)=\cl{\R(T)}\cap \cl{\R(T^*)}$, i.e. $\cl{\R(B)}=\cl{\R(T)}\ominus (\cl{\R(T)}\cap \cl{\R(T^*)})$.
\end{itemize}
\end{theorem}
\begin{proof}
The operator $T$ is CoR if and only if $(T-T^*)P_{\cl{\R(T)}\cap \cl{\R(T^*)}}=0$, i.e. $$\begin{bmatrix} (A-A^*)P & 0 \\ -B^* P & 0 \end{bmatrix} = 0.$$ From here the first statement of the theorem follows directly.

Suppose now that $T$ is a CoR operator.

\noindent (1) We already have $AP=A^*P$, and $\R(P)\subseteq \N(B^*)$. If $x\in \R(P)=\cl{\R(T)}\cap \cl{\R(T^*)}$ is arbitrary, then: $$T^* x = \begin{bmatrix} A^* & 0 \\ B^* & 0 \end{bmatrix} \begin{bmatrix} x \\ 0 \end{bmatrix} = \begin{bmatrix} A^*x \\ 0 \end{bmatrix} \in \cl{\R(T)}.$$ Since $T^* x\in \cl{\R(T)}\cap \cl{\R(T^*)}$, we have that $A^*x\in \cl{\R(T)}\cap \cl{\R(T^*)}$. This proves $PA^*P=A^*P$, but $A^*P=AP$, so $PAP=AP$ also. From here we also obtain $(I-P)A(I-P)=A(I-P)$, and $(I-P)A^*(I-P)=A^*(I-P)$, so $\cl{\R(T)}\ominus (\cl{\R(T)}\cap \cl{\R(T^*)})$ is also invariant for $A$ and $A^*$. The equality $B^* P =0$ implies $\R(B)\subseteq \cl{\R(T)}\ominus (\cl{\R(T)}\cap \cl{\R(T^*)})$. Finally, if $\R(T)$ is closed, we see that the subspaces $\R(T)\cap \R(T^*)$ and $\R(T)\ominus (\R(T)\cap \R(T^*))$ are invariant also for $AA^* + BB^*$ which is an isomorphism. Therefore $P\Delta P = \Delta P$.

%(3) From (1) we have that $T^*(\R(T)\cap \R(T^*))=\R(T)\cap \R(T^*)$, but also for $x\in \R(T)\cap \R(T^*)$ we have $T^* x = \begin{bmatrix}{c} A^* x\\ 0 \end{bmatrix}$. This shows that $A^*(\R(T)\cap \R(T^*))=\R(T)\cap \R(T^*)$, i.e. $\R(T)\cap \R(T^*)\subseteq \R(A^*)$. Since $AA^* + BB^*$ is an isomorphism which is completely reduced by the decomposition $\R(T)=(\R(T)\cap \R(T^*)) \oplus [ \R(T)\ominus (\R(T)\cap \R(T^*))]$, the reduction of $AA^* + BB^*$ on $\R(T)\cap \R(T^*)$ is also an isomorphisms. The operator $BB^*$ is equal to null-operator on this subspace, so $AA^*(\R(T)\cap \R(T^*))=\R(T)\cap \R(T^*)$, i.e. $A(\R(T)\cap \R(T^*))=\R(T)\cap \R(T^*)$. Thus $\R(T)\cap \R(T^*)\subseteq \cl{\R(A)}\cap \cl{\R(A^*)}$.

\noindent (2) If $x\in \cl{\R(T)}\ominus (\cl{\R(T)}\cap \cl{\R(T^*)})$ is such that $B^*x =0$, then $T^*x\in \cl{\R(T)}$, i.e. $T^*x\in \cl{\R(T)}\cap \cl{\R(T^*)}$. Therefore, $TT^* x\in \cl{\R(T)}\cap \cl{\R(T^*)}$, and so $0=\langle x, TT^*x\rangle = ||T^*x||^2$, giving $x=0$. Thus $\N(B^*)=\cl{\R(T)}\cap \cl{\R(T^*)}$.
\end{proof}

In order to give a formula for $(T+T^*)^\dag$ when $T$ is CoR, we first prove the following result regarding range additivity, explaining when does $(T+T^*)^\dag$ exist (see also \cite[Theorem 2.4]{Djikic2}).

\begin{theorem}\label{djt4} Let $T\in \B(\H)$ be a closed range CoR operator. Then $\R(TT^\dag - TT^\dag)$ is closed if and only if $\R(T) + \R(T^*)$ is closed if and only if $\R(T+T^*)$ is closed. In that case $\R(T) + \R(T^*)=\R(T+T^*)$.
\end{theorem}
\begin{proof}
The first equivalence follows from Lemma \ref{djl1}, so we prove the second equivalence.

Suppose first that $\R(T) + \R(T^*)$ is closed. Let $\H_1, \pi$ and $T_1$ be defined as in the proof of Theorem \ref{djt3}. Then $T_1$ is a closed range DR operator. Note that $\R(T_1)\oplus \R(T_1^*)=\pi(\R(T) + \R(T^*))$, and $\R(T)\cap \R(T^*)\subseteq \R(T) + \R(T^*)$, so using Theorem \ref{djt3} we have that $\R(T_1)\oplus \R(T_1^*)$ is also closed.  According to \cite[Theorem 3.10]{Arias2}, we have that $\R(T_1 + T_1^*)=\R(T_1)\oplus \R(T_1^*)$, so $\R(T_1 + T_1^*)$ is also closed. We can easily prove that $\R(T_1 + T_1^*)=\pi (\R(T + T^*))$, and since $T$ is CoR, $\R(T)\cap \R(T^*) = T(\R(T)\cap \R(T^*)) \subseteq \R(T+T^*)$ (Lemma \ref{djl2}). Thus, again from Theorem \ref{djt3} we get that $\R(T+T^*)$ is also closed.

Suppose now that $\R(T+T^*)$ is closed. From \cite[Corollary 2.1]{Djikic2} we have that $\R(T+T^*)=\R(T) + \R(T^*)$, so $\R(T)+\R(T^*)$ is also closed. This also proves the second statement of the theorem.
\end{proof}

Thus the range additivity $\R(T+T^*)=\R(T) + \R(T^*)$ which appears in \cite[Theorem 3.9 (ii)]{Deng2} is also present in the case when operators are DR and not necessarily SR. For matrices, this was noted in \cite[p. 1229]{Baks1}, but the technique used therein relies on notions which are not accessible in infinite-dimensional Hilbert spaces.

\begin{theorem}\label{djt6} If $T$ is a closed range CoR operator and if any of the (equivalent) conditions is satisfied: $\R(B)$ is closed, $\R(T+T^*)$ is closed, $\R(T)+\R(T^*)$ is closed, or $\R(TT^\dag - T^\dag T)$ is closed, then: \begin{equation}\label{djeq3}(T+T^*)^\dag = \begin{bmatrix} \frac{1}{2} A^* \Delta P & (B^*)^\dag \\ B^\dag & -B^\dag (A+A^*)(B^*)^\dag  \end{bmatrix},\end{equation}
where operators $A,B,\Delta$ and $P$ are defined as in the previous discussion.
\end{theorem}
\begin{proof}
Denote by $X$ the operator on the right in \eqref{djeq3}. By direct multiplication, we obtain: $$(T+T^*)X=\begin{bmatrix} \frac{1}{2}(A+A^*)A^*\Delta P + BB^\dag & (A+A^*)(B^*)^\dag - BB^\dag (A+A^*)(B^*)^\dag \\ \frac{1}{2} B^*A^*\Delta P & (B^*)(B^*)^\dag
\end{bmatrix}.$$
From Theorem \ref{djt5} we have $B^* P = 0$, $BB^\dag = I-P$, $P(B^*)^\dag = 0$, $P(A+A^*)=(A+A^*)P=2AP$, $AA^*P = (AA^* + BB^*)P$, $\Delta P = P\Delta P$. Hence:
\begin{eqnarray*}
 \frac{1}{2}(A+A^*)A^*\Delta P + BB^\dag & = &  \frac{1}{2}(A+A^*)PA^*P\Delta P + I-P \\ & = & AA^* P \Delta P + I- P \\ & = & (AA^* + BB^*)P\Delta P + I-P \\ & = & I,
\end{eqnarray*}
also $(A+A^*)(B^*)^\dag - BB^\dag (A+A^*)(B^*)^\dag = (A+A^*)(B^*)^\dag - (I-P)(A+A^*)(B^*)^\dag =0$, and $\frac{1}{2} B^*A^*\Delta P = \frac{1}{2} B^*PA^*P\Delta P = 0$. So we conclude: $$(T+T^*)X = \begin{bmatrix} I & 0 \\ 0 & P_{\R(B^*)} \end{bmatrix} : \begin{bmatrix} \R(T)  \\ \N(T^*)\end{bmatrix} \to \begin{bmatrix} \R(T)  \\ \N(T^*)\end{bmatrix}.$$ From Theorem \ref{djt4} and the proof of statement (2) in Theorem \ref{djt2} we have that $\R(T+T^*)=\R(T) + \R(T^*)=\R(T)\oplus \R(B^*)$. So $(T+T^*)X$ is the orthogonal projection onto $\R(T+T^*)$. It is also true that $X$ is self--adjoint. To see this, note that $\Delta$ is self--adjoint and that $A^*P=PA^*=PA^*P$ commutes with $(AA^* + BB^*)P=P(AA^* + BB^*)=P(AA^*+BB^*)P$, and so it commutes with $\Delta P$. Thus $A^*\Delta P = \Delta A^* P =\Delta A P =P\Delta A$. Hence, $X(T+T^*)$ is also the orthogonal projection onto $\R(T+T^*)$. This proves $X=(T+T^*)^\dag$.
\end{proof}

Formula \eqref{djeq3} generalizes the result from \cite[Theorem 3.9]{Deng2} regarding the formula for $(T+T^*)^\dag$, and we have $T(T+T^*)^\dag T = T - \frac{1}{2}P_{\R(T)\cap \R(T^*)}TP_{\R(T)\cap \R(T^*)}$, while $2T(T+T^*)^\dag T^* = 2T^*(T+T^*)^\dag T = P_{\R(T)\cap \R(T^*)}TP_{\R(T)\cap \R(T^*)}$. In fact, the last expression gives the parallel sum of $T$ and $T^*$ as defined in \cite{Corach2}, and in the same time the infimum of $T$ and $T^*$ with respect to the star partial order on $\B(\H)$ (see \cite{Djikic2}).

There are few results from \cite{Baks1} for DR matrices that can be easily proved for CoR operators in the Hilbert space setting. For example, \cite[Theorem 4, Theorem 5]{Baks1} are also true for CoR operators, and \cite[Theorem 8]{Baks1} can be extended using \cite[Theorem 2.4]{Djikic2}. However, we can not have elegant characterizations as the one in \cite[Theorem 1]{Baks1}, since the CoR class is not defined only by mutual positioning of the ranges of appropriate operators. When we make a transition from operators $T$ and $T^*$ to the orthogonal projections $P=P_{\cl{\R(T)}}$ and $Q=P_{\cl{\R(T^*)}}$, we lose the information of the way $T$ and $T^*$ act on these subspaces which determines whether $T$ is CoR.

\section{Products of orthogonal projections}

If $\mathfrak{A}$ and $\mathfrak{B}$  are two classes of operators from $\B(\H)$ there is a natural problem of characterizing operators which belong to the class $\mathfrak{A}\cdot \mathfrak{B} = \{A\cdot B\ |\ A\in \mathfrak{A},\ B\in \mathfrak{B}\}$, or to the class $\mathfrak{A}^\infty = \bigcup_k \mathfrak{A}^k$, where $\mathfrak{A}^k$ stands for $\mathfrak{A}\cdot \mathfrak{A}\cdot ... \cdot \mathfrak{A}$. Such problems are commonly known as factorization problems, and the reader is referred to \cite{Wu, Oikhberg, Corach3, Arias} for some prominent results and further reference on this subject.

Let $\mf{P}$ denote the class of all orthogonal projections from $\B(\H)$. We have the following results regarding the factors from $\mf{P}$.

\begin{theorem}\label{djt1}
Let $P,Q\in \mf{P}$ and $T=P_1P_2...P_k$ such that $P_1,P_2,...,P_k\in \{P,Q\}$. Then $T$ is a CoR operator.
\end{theorem}

\begin{proof}
Since $P$ and $Q$ are idempotents, we can exchange multiple consecutive appearances of $P$, i.e. $Q$, with only one $P$, i.e. $Q$. Thus we can suppose that $T=PQPQ...S$ or $T=QPQP...S$ where $S$ is equal to $P$ or $Q$.

Suppose that $T=PQPQ...S$. If $S=P$, then $T=T^*$ and the assertion follows. If $S=Q$, then $T^*=QPQP...P$, so $\cl{\R(T)}\subseteq \cl{\R(P)}=\R(P)$, $\cl{\R(T^*)}\subseteq \R(Q)$ and $\cl{\R(T)}\cap \cl{\R(T^*)}\subseteq \R(P)\cap \R(Q)$. On the other hand, it is clear that $\R(P)\cap \R(Q)\subseteq \R(T)\cap \R(T^*)\subseteq \cl{\R(T)}\cap \cl{\R(T^*)}$. Thus $\cl{\R(T)}\cap \cl{\R(T^*)}=\R(P)\cap \R(Q)$ and both $T$ and $T^*$ are equal to identity on $\cl{\R(T)}\cap \cl{\R(T^*)}$. The case when $P=QPQP...S$ is, of course, the same.
\end{proof}

\begin{corollary}\label{djc1} The class $\mf{P}^2$ belongs to the class of CoR operators.
\end{corollary}

\begin{proof}
Follows directly from Theorem \ref{djt1}.
\end{proof}

\begin{corollary}\label{djc3} If $P$ and $Q$ are orthogonal projections such that $\R(PQ\cdots PQ)$ is closed, then $\R(PQ\cdots PQ + QP\cdots QP)$ is closed if and only if $\R(PQ\cdots PQ) + \R(QP\cdots QP)$ is closed. In that case $\R(PQ\cdots PQ)+\R(QP\cdots QP)=\R(PQ\cdots PQ + QP\cdots QP)$. (Here $PQ...PQ$ and $QP...QP$ have the same length.)
\end{corollary}
\begin{proof}
Directly from Theorem \ref{djt1} and Theorem \ref{djt4}.
\end{proof}

Note that Corollary \ref{djc3} generalizes \cite[Corollary 4]{Baks2} in infinite-dimensional setting and for products of arbitrary length.

\begin{example}\label{djex1} Let us show that there exists $T\in \mf{P}^2$ such that $T$ is not in any of the classes described in Definition \ref{djd1}. Let $\H=\mathbb{C}^4$ and:
$$P=\begin{bmatrix}1 & 0 & 0 & 0 \\ 0 & 1 & 0 & 0 \\ 0 & 0 & 0 & 0 \\ 0 & 0 & 0 & 0 \end{bmatrix}, \quad \quad Q=\begin{bmatrix} \frac{3}{4} & 0 & \frac{\sqrt{3}}{4} & 0 \\ 0 & 1 & 0 & 0 \\ \frac{\sqrt{3}}{4} & 0 & \frac{1}{4} & 0 \\ 0 & 0 & 0 & 0 \end{bmatrix}.$$
Then $P,Q\in \mf{P}$, while for $T=PQ\in \mf{P}^2$ we have:
$$T=\begin{bmatrix} \frac{3}{4} & 0 & \frac{\sqrt{3}}{4} & 0 \\ 0 & 1 & 0 & 0 \\ 0& 0 & 0 & 0 \\ 0& 0 & 0 & 0\end{bmatrix},\quad \quad T^* = \begin{bmatrix} \frac{3}{4} & 0 & 0 & 0 \\ 0 & 1 & 0 & 0 \\ \frac{\sqrt{3}}{4} & 0 & 0 & 0 \\ 0 & 0 & 0 & 0 \end{bmatrix}.$$
We readily check that $T$ does not belong to any of the classes EP, DR, SR, co-EP, weak-EP.
\end{example}

\begin{example}\label{djex2}
We will show now that the class $\mf{P}^3$ is not contained in the CoR class. Let $\H=\mathbb{C}^4$ and:
$$P=\begin{bmatrix} 1& 0 & 0 & 0 \\ 0 & 1 & 0 & 0 \\ 0 & 0 & 0 & 0\\0 & 0&0 & 0\end{bmatrix},\quad Q=\begin{bmatrix}\frac{1}{2} & 0 & 0 & \frac{1}{2}\\ 0 & 1 & 0 & 0 \\ 0 & 0 & 0 & 0 \\ \frac{1}{2} & 0 & 0 & \frac{1}{2}\end{bmatrix},\quad R=\begin{bmatrix} 1 & 0 & 0 & 0 \\ 0 & \frac{3}{4} & 0 & \frac{\sqrt{3}}{4} \\ 0 & 0 & 1 & 0 \\ 0 & \frac{\sqrt{3}}{4} & 0 & \frac{1}{4}\end{bmatrix}.$$
Then $P,Q,R\in\mf{P}$, while for $T=PQR\in \mf{P}^3$ we have:
$$T=\begin{bmatrix}\frac{1}{2} & \frac{\sqrt{3}}{8} &0 & \frac{1}{8} \\ 0 & \frac{3}{4} & 0 & \frac{\sqrt{3}}{4} \\ 0 & 0 & 0 & 0 \\ 0 &0 & 0 & 0\end{bmatrix},\quad T^*=\begin{bmatrix} \frac{1}{2} & 0 & 0 & 0 \\ \frac{\sqrt{3}}{8} & \frac{3}{4} & 0 & 0 \\ 0 & 0 & 0 & 0 \\ \frac{1}{8} & \frac{\sqrt{3}}{4} & 0 & 0 \end{bmatrix}.$$
We can now check that $x=(1,0,0,0)\in \R(T)\cap \R(T^*)$, however $Tx\not = T^*x$, and so $T$ is not CoR.
\end{example}

%\begin{theorem}\label{djt1} Class $\mf{P}^\infty$ belongs to the class of CoR operators. Moreover, there exists $T\in \mf{P}^2$ such that $T$ does not belong to any of the classes EP, DR, SR, co-EP, weak-EP.
%\end{theorem}
%\begin{proof}
%We will prove the following statement: if $T=P_kP_{k-1}...P_1$, where $P_i$ are orthogonal projections, then $\R(T)\cap \R(T^*) =\R(P_1)\cap \R(P_2)\cap ... \cap \R(P_k)$. We will do this by induction on $k$, the number of orthogonal projections appearing in the factorization of $T$.
%
%If $k=1$, this is clear. Suppose that the statement is true for $k-1$, and we will prove that it is also true for $k$. If $T=P_kP_{k-1}...P_1$, then it is clear that $\R(P_1)\cap \R(P_2)\cap ... \cap \R(P_k)\subseteq \R(T)\cap \R(T^*)$. On the other hand, denoting by $Q=P_kP_{k-1}...P_{2}$ we have that $T=QP_1$, $T^*=P_1Q^*$ and by induction, $\R(Q)\cap \R(Q^*)=\R(P_k)\cap \R(P_{k-1})\cap ... \cap \R(P_2)$. %  By induction on $k$ we will prove that $\mf{P}^k$ belongs to the class CoR, for every $k$. If $k=1$ this is clear. Suppose that $\mf{P}^k$ belongs to CoR and let $T\in \mf{P}^{k+1}$. Then $T=P\cdot Q$ where $P$ is orthogonal projection,
%\end{proof}

%{\bf Acknowledgments.} The author is supported by Grant No. 174007 of the  Ministry of Education, Science and Technological Development, Republic of Serbia.

\end{document}